\UseRawInputEncoding
\documentclass{amsart}

\usepackage{geometry, amssymb, verbatim, esint}

\usepackage[colorlinks=true,urlcolor=blue, citecolor=red,linkcolor=blue,linktocpage,pdfpagelabels, bookmarksnumbered,bookmarksopen]{hyperref}
\usepackage[hyperpageref]{backref}

\numberwithin{equation}{section}

\newtheorem{lm}{Lemma}[section]
\newtheorem{prop}[lm]{Proposition}
\newtheorem{teo}[lm]{Theorem}
\newtheorem{coro}[lm]{Corollary}
\theoremstyle{definition}
\newtheorem{oss}[lm]{Remark}
\newtheorem*{ack}{Acknowledgments}

\author[Brasco]{Lorenzo Brasco}

\address[L.\ Brasco]{Dipartimento di Matematica e Informatica
	\newline\indent
	Universit\`a degli Studi di Ferrara
	\newline\indent
	Via Machiavelli 35, 44121 Ferrara, Italy}
\email{lorenzo.brasco@unife.it}

\title{Convex duality for principal frequencies}

\begin{document}

\begin{abstract}
We consider the sharp Sobolev-Poincar\'e constant for the embedding of $W^{1,2}_0(\Omega)$ into $L^q(\Omega)$. We show that such a constant exhibits an unexpected dual variational formulation, in the range $1<q<2$. Namely, this can be written as a convex minimization problem, under a divergence--type constraint.
This is particularly useful in order to prove lower bounds. The result generalizes what happens for the torsional rigidity (corresponding to $q=1$) and extends up to the case of the first eigenvalue of the Dirichlet-Laplacian (i.e. to $q=2$).
\end{abstract}


\maketitle

\begin{center}
\begin{minipage}{10cm}
\small
\tableofcontents
\end{minipage}
\end{center}

\section{Introduction}

\subsection{Overview}
Let $\Omega\subset\mathbb{R}^N$ be an open set, we denote by $W^{1,2}_0(\Omega)$ the closure of $C^\infty_0(\Omega)$ with respect to the Sobolev norm
\[
\|\varphi\|_{W^{1,2}(\Omega)}=\|\varphi\|_{L^2(\Omega)}+\|\nabla \varphi\|_{L^2(\Omega;\mathbb{R}^N)},\qquad \mbox{ for every } \varphi\in C^\infty_0(\Omega).
\] 
In what follows, we will always consider for simplicity sets with finite measure. This guarantees that
 we have at our disposal the Poincar\'e inequality
\[
C_\Omega\,\int_\Omega |\varphi|^2\,dx\le \int_\Omega |\nabla \varphi|^2\,dx,\qquad \mbox{ for every } \varphi\in W^{1,2}_0(\Omega).
\]
Consequently, the space $W^{1,2}_0(\Omega)$ can be equivalently endowed with the norm
\[
\|\varphi\|_{W^{1,2}_0(\Omega)}:=\|\nabla \varphi\|_{L^2(\Omega;\mathbb{R}^N)}.
\]
For $1\le q\le 2$, we consider the {\it generalized principal frequency}
\[
\lambda_1(\Omega;q)=\min_{\varphi\in W^{1,2}_0(\Omega)\setminus\{0\}} \frac{\displaystyle\int_\Omega |\nabla \varphi|^2\,dx}{\displaystyle\left(\int_\Omega |\varphi|^q\,dx\right)^\frac{2}{q}},
\]
considered also in \cite{KJ} and more recently in \cite{va, CR, Zu}, among others. The fact that the minimum above is attained in $W^{1,2}_0(\Omega)$ follows from the compactness of the embedding $W^{1,2}_0(\Omega)\hookrightarrow L^q(\Omega)$. The latter holds true since we are assuming that $\Omega$ has finite measure, see \cite[Theorem 2.8]{BS}.
\par
Two important cases deserve to be singled-out from the very beginning: $q=1$ and $q=2$.
In the first case, this quantity actually coincides with the reciprocal of the so-called {\it torsional rigidity} of $\Omega$
\[
\frac{1}{\lambda_1(\Omega;1)}=T(\Omega):=\max_{\varphi\in W^{1,2}_0(\Omega)\setminus\{0\}} \frac{\displaystyle\left(\int_\Omega |\varphi|\,dx\right)^2}{\displaystyle\int_\Omega |\nabla \varphi|^2\,dx}.
\]
For $q=2$, on the other hand, the quantity $\lambda_1(\Omega;2)$ is nothing but the {\it first eigenvalue of the Dirichlet-Laplacian} or {\it principal frequency} of $\Omega$. This is the smallest real number $\lambda$ such that the Helmholtz equation
\[
-\Delta u=\lambda\,u,\qquad \mbox{ in }\Omega,
\]
admits a nontrivial weak solution in $W^{1,2}_0(\Omega)$.
For simplicity, we will simply denote this quantity by $\lambda_1(\Omega)$.
\vskip.2cm\noindent
In general, the quantity $\lambda_1(\Omega;q)$ can not be explicitly computed. It is then quite useful to seek for (possibly sharp) estimates in terms of geometric quantities of the set $\Omega$. Some particular instances of results in this direction are given by: 
\begin{itemize}
\item the {\it Faber-Krahn inequality} (see for example \cite[Theorem 2]{KJ})
\[
\lambda_1(\Omega;q)\ge \left(\frac{\lambda_1(B;q)}{|B|^{1-\frac{2}{N}-\frac{2}{q}}}\right)\,|\Omega|^{1-\frac{2}{N}-\frac{2}{q}},
\]
which is valid for every open set $\Omega\subset\mathbb{R}^N$ with finite measure. Here $B$ is any $N-$dimensional open ball;
\vskip.2cm
\item the {\it Hersch-Makai--type inequality} (see \cite[Theorem 1.1]{BM})
\[
\lambda_1(\Omega;q)\ge \left(\frac{\pi_{2,q}}{2}\right)^2\, \frac{|\Omega|^\frac{q-2}{q}}{R_\Omega^2},
\]
which is valid for every open bounded {\it convex} set $\Omega\subset\mathbb{R}^N$. Here $R_\Omega$ is the {\it inradius}, i.e. the radius of a largest ball contained in $\Omega$ and $\pi_{2,q}$ is the constant defined by
\[
\pi_{2,q}=\inf_{\varphi\in W^{1,2}_0((0,1))\setminus\{0\}}\frac{\|\varphi'\|_{L^2((0,1))}}{\|\varphi\|_{L^q((0,1))}}.
\]
This inequality is the extension to the range $1\le q\le 2$ of \cite[equation (3')]{Ma} by Makai (case $q=1$) and of \cite[Th\'eor\`eme 8.1]{He2} by Hersch (case $q=2$);
\vskip.2cm
\item the {\it P\'olya--type inequality} (see \cite[Main Theorem]{Br})
\[
\lambda_1(\Omega;q)\le \left(\frac{\pi_{2,q}}{2}\right)^2\,\left(\frac{P(\Omega)}{|\Omega|^{\frac{1}{2}+\frac{1}{q}}}\right)^2,
\]
which is again valid for open bounded convex sets. Here $P(\Omega)$ stands for the perimeter of $\Omega$. This inequality generalizes the original result by P\'olya \cite{Po} for the cases $q=1$ and $q=2$.  
\end{itemize}
We point out that all the previous estimates are sharp. All the exponents appearing above are of course dictated by scale invariance.
\par
As a general rule, we can assert that lower bounds on $\lambda_1(\Omega;q)$ are harder to obtain with respect to upper bounds, since every generalized principal frequency is defined as an {\it infimum}. It would then be interesting to investigate whether $\lambda_1(\Omega;q)$ admits a sort of ``dual'' equivalent formulation, in terms of a supremum. This is the main goal of the present paper.

\subsection{Towards duality}
\label{sec:torsion}
At this aim, it is interesting to have a closer look at the case $q=1$. It is well-known that the torsional rigidity can be equivalently rewritten as an unconstrained concave maximization problem, i.e.
\begin{equation}
\label{unconstrained}
\max_{\varphi\in W^{1,2}_0(\Omega)}\left\{2\,\int_\Omega \varphi\,dx-\int_\Omega |\nabla \varphi|^2\,dx\right\}=T(\Omega).
\end{equation}
As such, it admits in a natural way a dual convex minimization problem
\begin{equation}
\label{dualT}
\begin{split}
\min_{\phi\in L^2(\Omega;\mathbb{R}^N)}& \left\{\int_\Omega |\phi|^2\,dx\, :\, -\mathrm{div\,}\phi=1 \mbox{ in }\Omega\right\}\\
&=\max_{\varphi\in W^{1,2}_0(\Omega)}\left\{2\,\int_\Omega \varphi\,dx-\int_\Omega |\nabla \varphi|^2\,dx\right\}=T(\Omega),
\end{split}
\end{equation}
which gives yet another equivalent definition of torsional rigidity. Here the divergence constraint has to be intended in distributional sense, i.e. 
\begin{equation}
\label{divergenzia}
\int_\Omega \langle \phi,\nabla \varphi\rangle\,dx=\int_\Omega \varphi\,dx,\qquad \mbox{ for every } \varphi\in C^1_0(\Omega).
\end{equation}
By means of a standard density argument, it is easily seen that we can enlarge the class of competitors in \eqref{divergenzia} to the whole $W^{1,2}_0(\Omega)$, since $\phi\in L^2(\Omega;\mathbb{R}^N)$.
\par
For a better understanding of the contents of this paper, it may be useful to recall the proof of \eqref{dualT}. At first, one observes that for every admissible vector fields and every $\varphi\in W^{1,2}_0(\Omega)$, we have 
\[
2\,\int_\Omega \varphi\,dx-\int_\Omega |\nabla \varphi|^2\,dx=2\,\int_\Omega \langle \phi,\nabla \varphi\rangle\,dx-\int_\Omega |\nabla \varphi|^2\,dx,
\]
by virtue of \eqref{divergenzia}.
We can now use Young's inequality 
\[
2\,\langle \phi,\nabla \varphi\rangle-|\nabla \varphi|^2\le |\phi|^2.
\] 
By integrating this inequality, from the identity above we get
\[
2\,\int_\Omega \varphi\,dx-\int_\Omega |\nabla \varphi|^2\,dx\le \int_\Omega |\phi|^2\,dx.
\] 
The arbitrariness of $\varphi$ and $\phi$ automatically gives
\[
\max_{\varphi\in W^{1,2}_0(\Omega)}\left\{2\,\int_\Omega \varphi\,dx-\int_\Omega |\nabla \varphi|^2\,dx\right\}\le \min_{\phi\in L^2(\Omega;\mathbb{R}^N)}  \left\{\int_\Omega |\phi|^2\,dx\, :\, -\mathrm{div\,}\phi=1 \mbox{ in }\Omega\right\}.
\]
On the other hand, by taking $\varphi=w$ to be the optimal function for the problem on the left-hand side, this satisfies the relevant Euler-Lagrange equation. The latter is given by
\[
-\Delta w=1,\qquad \mbox{ in }\Omega.
\]
Thus $\phi_0=\nabla w$ is an admissible vector field and we have 
\[
\int_\Omega |\phi_0|^2\,dx=2\,\int_\Omega \langle \nabla w,\phi_0\rangle\,dx-\int_\Omega |\nabla w|^2 =2\,\int_\Omega w\,dx-\int_\Omega |\nabla w|^2\,dx.
\]
This proves that 
\[
\min_{\phi\in L^2(\Omega;\mathbb{R}^N)}  \left\{\int_\Omega |\phi|^2\,dx\, :\, -\mathrm{div\,}\phi=1 \mbox{ in }\Omega\right\}\le \max_{\varphi\in W^{1,2}_0(\Omega)}\left\{2\,\int_\Omega \varphi\,dx-\int_\Omega |\nabla \varphi|^2\,dx\right\},
\]
as well. Thus \eqref{dualT} holds true and we also have obtained that the unique (by strict convexity) minimal vector field $\phi_0\in L^2(\Omega;\mathbb{R}^N)$ must be of the form
\[
\phi_0=\nabla w,
\]
with $w$ being the unique $W^{1,2}_0(\Omega)$ solution of $-\Delta w=1$. In conclusion, getting back to the notation $\lambda_1(\Omega;1)$, we obtain the following dual characterization of the relevant generalized principal frequency
\begin{equation}
\label{dual1}
\frac{1}{\lambda_1(\Omega;1)}=\min_{\phi\in L^2(\Omega;\mathbb{R}^N)}  \left\{\int_\Omega |\phi|^2\,dx\, :\, -\mathrm{div\,}\phi=1 \mbox{ in }\Omega\right\}.
\end{equation}

\subsection{Main results}

The main result of the present paper asserts that the dual characterization \eqref{dual1} {\it is not} an isolated exception. Actually, it is possible to prove that 
\[
\frac{1}{\lambda_1(\Omega;q)},
\] 
coincides with the {\it minimum} of a constrained convex minimization problem, for the whole range $1\le q\le 2$. The deep reason behind this result is a {\it hidden convex structure} of the problem which defines $\lambda_1(\Omega;q)$, see Remark \ref{oss:hidden} below. Such a convex structure, which apparently is still not very popular, fails for $q>2$ and this explains why our result has $1\le q\le 2$ as the natural range of validity. 
\par 
In order to precisely state the result, we need at first to introduce the following convex lower semicontinuous function $G_q:\mathbb{R}\times\mathbb{R}^N\to [0,+\infty]$, defined for $1<q\le 2$ by: 
\begin{equation}
\label{G}
G_q(s,\xi)=\left\{\begin{array}{rl}
\dfrac{|\xi|^q}{|s|^{q-1}},& \mbox{ if } \xi\in\mathbb{R}^N,\, s<0,\\
&\\
0, & \mbox{ if } s=0,\, \xi=0,\\
&\\
+\infty,& \mbox{ otherwise}.
\end{array}
\right.
\end{equation}
We then distinguish between the cases $q<2$ and $q=2$.
\begin{teo}[Sub-homogeneous case]
\label{teo:sub}
Let $1<q<2$ and let $\Omega\subset\mathbb{R}^N$ be an open set, with finite measure. If we set 
\[
\mathcal{A}(\Omega)=\Big\{(f,\phi)\in L^1_{\rm loc}(\Omega)\times L^2_{\rm loc}(\Omega;\mathbb{R}^N)\, :\, -\mathrm{div\,} \phi+f\ge 1 \mbox{ in }\Omega\Big\},
\]
then we have
\begin{equation}
\label{formulab1}
\frac{1}{\lambda_1(\Omega;q)}=(q-1)^{(q-1)\frac{2}{q}}\inf_{(f,\phi)\in \mathcal{A}(\Omega)} \Big\|G_q(f,\phi)\Big\|_{L^\frac{2}{2-q}(\Omega)}^\frac{2}{q},
\end{equation}
where $G_q$ is defined in \eqref{G}.
Moreover, if $w\in W^{1,2}_0(\Omega)$ denotes the unique positive solution of 
\[
\max_{\varphi\in W^{1,2}_0(\Omega)}\left\{\frac{2}{q}\,\int_\Omega |\varphi|^q\,dx-\int_\Omega |\nabla \varphi|^2\,dx\right\},
\]
we get that the pair $(f_0,\phi_0)$ defined by
\[
\phi_0=\frac{\nabla w}{w^{q-1}}\qquad \mbox{ and }\qquad f_0=-(q-1)\,\frac{|\nabla w|^2}{w^{q}},
\]
is a minimizer for the problem in \eqref{formulab1}.
\end{teo}
\begin{oss}
Observe that the previous result is perfectly consistent with the case $q=1$. Indeed, by formally taking the limit as $q$ goes to $1$ in the statement above, the role of the dual variable $f$ becomes immaterial and we get back \eqref{dualT}, together with the optimality condition $\phi_0=\nabla w$.
\end{oss}
For the limit case $q=2$, corresponding to the first eigenvalue of the Dirichlet-Laplacian, we have the following dual characterization.
\begin{teo}[Homogeneous case]
\label{teo:homo}
Let $\Omega\subset\mathbb{R}^N$ be an open set, with finite measure. If we set 
\[
\mathcal{A}(\Omega)=\Big\{(f,\phi)\in L^1_{\rm loc}(\Omega)\times L^2_{\rm loc}(\Omega;\mathbb{R}^N)\, :\, -\mathrm{div\,} \phi+f\ge 1 \mbox{ in }\Omega\Big\},
\]
then we have
\begin{equation}
\label{formulab2}
\frac{1}{\lambda_1(\Omega)}=
\inf_{(f,\phi)\in \mathcal{A}(\Omega)}\Big\|G_2(f,\phi)\Big\|_{L^\infty(\Omega)},
\end{equation}
where $G_2$ is defined in \eqref{G}. 
Moreover, if we denote by $U\in W^{1,2}_0(\Omega)$ any positive first eigenfunction of $\Omega$,
we get that the pair $(f_0,\phi_0)$ defined by
\[
\phi_0=\frac{1}{\lambda_1(\Omega)}\,\frac{\nabla U}{U}\qquad \mbox{ and }\qquad f_0=-\frac{1}{\lambda_1(\Omega)}\,\frac{|\nabla U|^2}{U^2},
\]
is a minimizer for the problem in \eqref{formulab2}.
\end{teo}

\begin{oss}
It may be worth recalling that the existence of a (sort of) dual formulation for $\lambda_1$ is not a complete novelty. A related result can be traced back in the literature and attributed to the fundamental contributions of Protter and Hersch.
 This is called {\it maximum principle for $\lambda_1$} and reads as follows
\[
\lambda_1(\Omega)=\max_{\phi} \inf_{x\in\Omega} \Big[\mathrm{div}\phi(x)-|\phi(x)|^2\Big],
\]
under suitable regularity assumptions on $\Omega$ and on the admissible vector fields. It is not difficult to see that 
\[
\phi_0(x)=-\frac{\nabla U}{U},
\]
is a maximizer for the previous problem, at least formally. Here $U$ is again any positive first eigenfunction of $\Omega$.
We refer to the paper \cite{He1} by Hersch for a presentation of this result and for a detailed discussion about its physical interpretation.
\end{oss}

\begin{oss}
As a last observation, we wish to point out the interesting papers \cite{FL} and \cite{FS}, where yet another equivalent characterization for the torsional rigidity $T(\Omega)$ is obtained, when $\Omega\subset\mathbb{R}^2$ is a simply connected open set. Such a characterization is in terms of a minimization problem among holomorphic functions (see \cite[Theorem 1.2]{FL}) and thus it is suitable for giving upper bounds on $T(\Omega)$ (see \cite{FS}).
\end{oss}

\subsection{Plan of the paper}
We start by exposing some preliminary facts in Section \ref{sec:2}. In Section \ref{sec:3} we consider a certain convex function and show that its Legendre-Fenchel transform is related to the function $G_q$ above. The core of the paper is Section \ref{sec:4}, where Proposition \ref{prop:riformula} permits to rewrite the value $\lambda_1(\Omega;q)$ as an unconstrained {\it concave} maximization problem, exactly as in the case of the torsional rigidity. We can then prove our main results in Section \ref{sec:5}. Finally, in the last section we briefly show some applications of our results to geometric estimates for principal frequencies.

\begin{ack}
We wish to thank Francesco Maggi for first introducing us to the hidden convexity of the Dirichlet integral. We also thank Rafael Benguria, Guillaume Carlier and Eleonora Cinti for some comments on a preliminary version of the paper. 
\end{ack}

\section{Preliminaries}
\label{sec:2}

We first recall that it is possible to rewrite the minimization problem which defines $\lambda_1(\Omega;q)$ as an unconstrained optimization problem, in the regime $1\le q<2$. This generalizes formula \eqref{unconstrained}.
The proof is standard, we include it for completeness.
\begin{prop}
\label{lm:base}
Let $1\le q<2$ and let $\Omega\subset\mathbb{R}^N$ be an open set, with finite measure. Then we have
\[
\max_{\varphi\in W^{1,2}_0(\Omega)}\left\{\frac{2}{q}\,\int_\Omega |\varphi|^q\,dx-\int_\Omega |\nabla \varphi|^2\,dx\right\}=\frac{2-q}{q}\,\left(\frac{1}{\lambda_1(\Omega;q)}\right)^{\frac{q}{2-q}}.
\]
Moreover, the maximization problem on the left-hand side admits a unique non-negative solution $w$, which has the following properties
\[
w\in L^\infty(\Omega) \qquad \mbox{ and }\qquad \frac{1}{w}\in L^\infty_{\rm loc}(\Omega).
\] 
\end{prop}
\begin{proof}
Existence of a maximizer follows by a standard application of the Direct Method in the Calculus of Variations. The fact that a non-negative maximizer exists is a consequence of the fact that the functional is even, thus we can always replace $\varphi$ by $|\varphi|$ without decreasing the energy.
\par
We also observe that for $\varphi\in W^{1,2}_0(\Omega)\setminus\{0\}$ and $t>0$, the quantity
\[
\frac{2}{q}\,t^q\,\int_\Omega |\varphi|^q\,dx-t^2\,\int_\Omega |\nabla \varphi|^2\,dx,
\]
is strictly positive for $t$ sufficiently small. This shows that
\[
\max_{\varphi\in W^{1,2}_0(\Omega)}\left\{\frac{2}{q}\,\int_\Omega |\varphi|^q\,dx-\int_\Omega |\nabla \varphi|^2\,dx\right\}>0,
\]
and thus $\varphi\equiv 0$ can not be a maximizer. Observe that the same argument, together with the locality of the functional, imply that for any maximizer $w$ we must have $w\not\equiv0$ on every connected component of $\Omega$.
By coupling this information with the optimality condition, we get that any non-negative maximizer $w$ must be a nontrivial weak solution of the Euler-Lagrange equation
\[
-\Delta w=w^{q-1},\qquad \mbox{ in }\Omega.
\]
In particular, $w$ is a weakly superharmonic function and
by the minimum principle, we get that $1/w\in L^\infty_{\rm loc}(\Omega)$. The fact that $w\in L^\infty(\Omega)$ follows from standard Ellipic Regularity. 
\par
Finally, uniqueness of the positive maximizer can be found in \cite[Lemma 2.2]{BFR}, where the uniqueness result of \cite[Theorem 1]{BO} is extended to the case of open sets, not necessarily smooth.
\vskip.2cm\noindent
In order to prove the claimed equality between the extremum values, it is sufficient to exploit the different homogeneities of the two integrals and the fact that the maximum problem is equivalently settled on $W^{1,2}_0(\Omega)\setminus\{0\}$. We then have 
\[
\begin{split}
\max_{\varphi\in W^{1,2}_0(\Omega)}\left\{\frac{2}{q}\,\int_\Omega |\varphi|^q\,dx-\int_\Omega |\nabla \varphi|^2\,dx\right\}=\max_{\varphi\in W^{1,2}_0(\Omega)\setminus\{0\}, t>0}\left\{\frac{2}{q}\,t^q\,\int_\Omega |\varphi|^q\,dx-t^2\,\int_\Omega |\nabla \varphi|^2\,dx\right\}.
\end{split}
\]
It is easily seen that, for every $\varphi\in W^{1,2}_0(\Omega)\setminus\{0\}$ the function
\[
t\mapsto \frac{2}{q}\,t^q\,\int_\Omega |\varphi|^q\,dx-t^2\,\int_\Omega |\nabla \varphi|^2\,dx,
\]
is maximal for 
\[
t_0=\left(\frac{\displaystyle\int_\Omega |\varphi|^q\,dx}{\displaystyle\int_\Omega |\nabla \varphi|^2\,dx}\right)^\frac{1}{2-q}.
\]
With such a choice of $t$, we get
\[
\frac{2}{q}\,t_0^q\,\int_\Omega |\varphi|^q\,dx-t_0^2\,\int_\Omega |\nabla \varphi|^2\,dx=\left(\frac{\displaystyle\left(\int_\Omega |\varphi|^q\,dx\right)^\frac{2}{q}}{\displaystyle\int_\Omega |\nabla \varphi|^2\,dx}\right)^\frac{q}{2-q}\, \frac{2-q}{q}.
\]
By recalling the definition of $\lambda_1(\Omega;q)$, we get the desired conclusion.
\end{proof}
We also record the following technical result: this will be useful somewhere in the paper. More sophisticated results about the dependence on $q$ of the quantity $\lambda_1(\Omega;q)$ can be found in \cite[Theorem 1]{AFI} and \cite{Er}. 
\begin{lm}
\label{lm:limite}
Let $\Omega\subset\mathbb{R}^N$ be an open set, with finite measure. Then we have
\[
\lim_{q\to 1^+} \lambda_1(\Omega;q)=\lambda_1(\Omega;1)\qquad \mbox{ and }\qquad \lim_{q\to 2^-} \lambda_1(\Omega;q)=\lambda_1(\Omega).
\]
\end{lm}
\begin{proof}
For every $\varphi\in W^{1,2}_0(\Omega)\setminus\{0\}$, we have by H\"older's inequality
\[
\frac{\displaystyle\int_\Omega |\nabla \varphi|^2\,dx}{\displaystyle\left(\int_\Omega |\varphi|^q\,dx\right)^\frac{2}{q}}\ge |\Omega|^{1-\frac{2}{q}}\, \frac{\displaystyle\int_\Omega |\nabla \varphi|^2\,dx}{\displaystyle\int_\Omega |\varphi|^2\,dx}\ge |\Omega|^{1-\frac{2}{q}}\,\lambda_1(\Omega).
\]
By taking the infimum over $\varphi$, this leads to 
\[
\lambda_1(\Omega;q)\ge |\Omega|^{1-\frac{2}{q}}\,\lambda_1(\Omega).
\]
On the other hand, if $U\in W^{1,2}_0(\Omega)$ is any minimizer for $\lambda_1(\Omega)$, then we have 
\[
\lambda_1(\Omega;q)\le \frac{\displaystyle\int_\Omega |\nabla U|^2\,dx}{\displaystyle\left(\int_\Omega |U|^q\,dx\right)^\frac{2}{q}}=\frac{\displaystyle\int_\Omega |U|^2\,dx}{\displaystyle\left(\int_\Omega |U|^q\,dx\right)^\frac{2}{q}}\,\lambda_1(\Omega).
\]
The last two displays eventually prove the desired result for $q$ converging to $2$. The other result can be proved in exactly the same way.
\end{proof}
\begin{oss}
The assumption on the finiteness of the measure is sufficient, but in general not necessary, for the previous result to hold. However, as observed in \cite[Remark 2.2]{Br}, for a general open set $\Omega\subset\mathbb{R}^N$ it may happen that 
\[
\limsup_{q\to 2^-} \lambda_1(\Omega;q)<\lambda_1(\Omega).
\]
\end{oss}

\section{A convex function}
\label{sec:3}

In order to prove the main result of this paper, we will need to study a particular convex function $F_q:\mathbb{R}\times\mathbb{R}^{N}\to[0,+\infty]$ and its {\it Legendre-Fenchel transform}
\[
F_q^*(s,\xi)=\sup_{(t,x)\in\mathbb{R}\times \mathbb{R}^N} \Big[s\,t+\langle \xi,x\rangle-F_q(t,x)\Big].
\] 
We refer to the classical monograph \cite{Ro} for the basic properties of this transform.
\begin{lm}
\label{lm:fenchel}
Let $1<q<2$, we consider the convex lower semicontinuous function $F:\mathbb{R}\times \mathbb{R}^N\to \mathbb{R}\cup \{+\infty\}$ defined by
\[
F_q(t,x)=\left\{\begin{array}{rl}
|x|^2\,t^{\frac{2}{q}-2},& \mbox{ if } x\in\mathbb{R}^N,\, t>0,\\
0, & \mbox{ if } t=0,\, x=0,\\
+\infty,& \mbox{ otherwise}. 
\end{array}
\right.
\]
Then its Legendre-Fenchel transform is given by the convex lower semicontinuous function
\[
F^*_q(s,\xi)=\left\{\begin{array}{rl}
\alpha_q\,|\xi|^\frac{2\,q}{2-q}\,|s|^\frac{2\,(1-q)}{2-q},& \mbox{ if } \xi\in\mathbb{R}^N,\, s<0,\\
0, & \mbox{ if } s=0,\, \xi=0,\\
+\infty,& \mbox{ otherwise},
\end{array}
\right.
\]
where the constant $\alpha_q$ is given by
\[
\alpha_q=\frac{2-q}{2\,q}\,\left(\frac{q-1}{q}\right)^\frac{2\,(q-1)}{2-q}\,\left(\frac{1}{2}\right)^\frac{q}{2-q}.
\]
\end{lm}
\begin{proof}
We divide the proof in various parts, according to the claim that we are going to prove. 
\vskip.2cm\noindent
{\bf Lower semicontinuity.} In order to verify the semicontinuity of $F_q$, we need to prove that the epigraph
\[
\mathrm{epi\,}(F_q)=\Big\{((t,x);s)\in \mathbb{R}^{N+1}\times\mathbb{R}\, :\, F_q(t,x)\le \lambda\Big\},
\] 
is a closed set. We take $\{((t_n,x_n);s_n)\}_{n\in\mathbb{N}}\subset \mathrm{epi\,}(F_q)$ such that
\[
\lim_{n\to\infty} t_n=t,\qquad \lim_{n\to\infty} x_n=x,\qquad \lim_{n\to\infty} s_n=s.
\]
By using the definition of $F_q$ and that of epigraph, the fact that
\begin{equation}
\label{sottolivello}
F_q(t_n,x_n)\le s_n,
\end{equation}
automatically entails that
\[
\{(t_n,x_n)\}_{n\in\mathbb{N}}\subset \Big((0,+\infty)\times\mathbb{R}^N \Big)\cup \{(0,0)\}.
\]
This in particular implies that the limit point $t$ is such that $t\ge 0$. The same can be said for $s$, since $F_q$ always assume positive values. 
\par
We now observe that if $t>0$, we would have $t_n>0$ for $n$ large enough. In this case, we can simply pass to the limit in \eqref{sottolivello} and get
\[
F_q(t,x)=|x|^2\,t^{\frac{2}{q}-2}=\lim_{n\to\infty} |x_n|^2\,t_n^{\frac{2}{q}-2}\le \lim_{n\to\infty} s_n=s,
\] 
thus proving that $((t,x);s)\in \mathrm{epi\,}(F_q)$.
\par
Let us now suppose that $t=0$ and assume by contradiction that $((0,x);s)\not \in \mathrm{epi\,}(F_q)$. This means that
\[
F_q(0,x)>s.
\]
By recalling that $s\ge 0$ and that $F_q(0,0)=0$, this would automatically gives that $x\not =0$. On the other hand, by \eqref{sottolivello}, we get that 
\[
\mbox{ either }\qquad t_n=0 \mbox{ and } x_n=0\qquad \mbox{ or }\qquad t_n>0\mbox{ and } |x_n|^2\le s_n\,t_n^{2-\frac{2}{q}}.
\]
This entails that
\[
x=\lim_{n\to\infty} x_n=0,
\]
which gives a contradiction. This finally proves that the epigraph is closed.
\vskip.2cm\noindent
{\bf Convexity.} We need to prove that for every $t_0,t_1\in\mathbb{R}$, $x_0,x_1\in\mathbb{R}^N$ and $\lambda\in [0,1]$, we have
\begin{equation}
\label{convexity}
F_q(\lambda\,t_0+(1-\lambda)\,t_1,\lambda\,x_0+(1-\lambda)\,x_1)\le \lambda\,F_q(t_0,x_0)+(1-\lambda)\,F_q(t_1,x_1).
\end{equation}
We observe that for $t_0,t_1\le 0$, every $x_0,x_1\in\mathbb{R}^N\setminus\{0\}$ and every $\lambda\in[0,1]$, we trivially have \eqref{convexity}, since both terms on the right-hand side are equal to $+\infty$. We are thus confined to prove \eqref{convexity} for 
\[
(t_0,x_0),(t_1,x_1)\in \Big((0,+\infty)\times\mathbb{R}^N \Big)\cup \{(0,0)\}.
\] 
Moreover, if at least one between $(t_0,x_0)$ and $(t_1,x_1)$ coincides with $(0,0)$, then again the desired inequality follows by a straighforward computation. Finally, we can assume that  
\[
(t_0,x_0),(t_1,x_1)\in (0,+\infty)\times\mathbb{R}^N.
\]
We introduce the function 
\[
F_2(t,x)=|x|^2\,t^{-1},\qquad \mbox{ if } x\in\mathbb{R}^N,\, t>0,
\]
and we observe that for every $1<q<2$ we have
\[
F_q(t,x)=F_2\left(t^{2-\frac{2}{q}},x\right),\qquad \mbox{ for } (t,x)\in (0,+\infty)\times\mathbb{R}^N.
\]
By using that $t\mapsto F_2(t,x)$ is decreasing, that $t\mapsto t^{2-\frac{2}{q}}$ is concave (since $1<q<2$) and that $(t,x)\mapsto F_2(t,x)$ is convex (see for example \cite[Lemma 5.17]{OTAM}), we get
\[
\begin{split}
F_q(\lambda\,t_0+(1-\lambda)\,t_1,\lambda\,x_0+(1-\lambda)\,x_1)&=F_2\left((\lambda\,t_0+(1-\lambda)\,t_1)^{2-\frac{2}{q}},\lambda\,x_0+(1-\lambda)\,x_1\right)\\
&\le F_2\left(\lambda\,t_0^{2-\frac{2}{q}}+(1-\lambda)\,t_1^{2-\frac{2}{q}},\lambda\,x_0+(1-\lambda)\,x_1\right)\\
&\le \lambda\,F_2\left(t_0^{2-\frac{2}{q}},x_0\right)+(1-\lambda)\,F_2\left(t_1^{2-\frac{2}{q}},x_1\right)\\
&=\lambda\,F_q(t_0,x_0)+(1-\lambda)\,F_q(t_1,x_1),
\end{split}
\]
as desired.
\vskip.2cm\noindent
{\bf Computation of $F^*_q$.} We now come to the computation of the Legendre-Fenchel transform. This is lengthy but elementary. We first observe that $F_q$ is positively $2/q-$homogeneous, that is for every $\tau>0$ we have
\[
F_q(\tau\,t,\tau\,x)=\tau^\frac{2}{q}\,F_q(t,x),\qquad \mbox{ for every } (t,x)\in \mathbb{R}\times\mathbb{R}^N.
\]
Correspondingly, $F^*_q$ will be positively $2/(2-q)-$homogeneous, by standard properties of the Legendre-Fenchel transform. Thanks to this remark, it is sufficient to compute for $\xi\in\mathbb{R}^N\setminus\{0\}$
\[
F^*_q(-1,\xi),\qquad F^*_q(0,0),\qquad F^*_q(0,\xi)\qquad\mbox{ and }\qquad F^*_q(1,\xi).
\]
By definition, we have 
\[
\begin{split}
F^*_q(s,\xi)&=\sup_{(t,x)\in\mathbb{R}\times\mathbb{R}^N} \Big[t\,s+\langle x,\xi\rangle-F_q(t,x)\Big]\\
&=\sup_{t\ge 0,\,x\in\mathbb{R}^N} \Big[t\,s+|x|\,|\xi|-F_q(t,x)\Big]\\
&=\sup_{(t,m)\in \mathcal{E}} \left[t\,s+m\,|\xi|-m^2\,t^{\frac{2}{q}-2}\right],
\end{split}
\]
where we set\footnote{For notational simplicity, we use the convention that $m^2\,t^{\frac{2}{q}-2}=0$ when both $t=0$ and $m=0$.} 
\[
\mathcal{E}=\Big\{(t,m)\in \mathbb{R}\times\mathbb{R}\, :\, t>0,\, m\ge 0\Big\}\cup \{(0,0)\}.
\]
We thus easily get for $\xi\in\mathbb{R}^N$
\[
F^*_q(1,\xi)=\sup_{(t,m)\in \mathcal{E}} \left[t+m\,|\xi|-m^2\,t^{\frac{2}{q}-2}\right]=+\infty,
\]
and
\[
F^*_q(0,0)=\sup_{(t,m)\in \mathcal{E}} \left[-m^2\,t^{\frac{2}{q}-2}\right]=0.
\]
Moreover, for $\xi\not =0$ we have
\[
F^*_q(0,\xi)=\sup_{(t,m)\in \mathcal{E}} \left[m\,|\xi|-m^2\,t^{\frac{2}{q}-2}\right]=+\infty,
\]
as can be seen by taking
\[
t=n\in\mathbb{N}\qquad \mbox{ and }\qquad m=\frac{1}{2}\,|\xi|\,n^{2-\frac{2}{q}}, 
\]
and letting $n$ go to $+\infty$.
We are left with computing for $\xi\in\mathbb{R}^N\setminus\{0\}$ 
\[
F^*_q(-1,\xi)=\sup_{(t,m)\in \mathcal{E}} \left[-t+m\,|\xi|-\frac{1}{2}\,m^2\,t^{\frac{2}{q}-2}\right].
\]
We make a preliminary observation: if we take $(t,m)\in\mathcal{E}$ such that 
\[
m>0\qquad \mbox{ and }\qquad t=\frac{|\xi|}{2}\,m,
\]
we get 
\[
F^*_q(-1,\xi)\ge \frac{|\xi|}{2}\,m-\frac{1}{2}\,\left(\frac{|\xi|}{2}\right)^{\frac{2}{q}-2}\, m^\frac{2}{q}
\]
and the last quantity can be made strictly positive, for $m>0$ small enough, thanks to the fact that $2/q>1$. On the contrary, every point $(t,0)\in\mathcal{E}$ can not be a maximizer for the problem which defines $F^*_q(1,-\xi)$, since on these points
\[
-t+m\,|\xi|-\frac{1}{2}\,m^2\,t^{\frac{2}{q}-2}=-t\le 0.
\]
This simple observation implies that we can rewrite the maximization problem for $F^*_q(1,-\xi)$ as
\[
F^*_q(-1,\xi)=\sup_{t>0,m>0} \left[-t+m\,|\xi|-\frac{1}{2}\,m^2\,t^{\frac{2}{q}-2}\right].
\]
In order to explicitly compute this, we will exploit the homogeneity of the function $(t,m)\mapsto m^2\,t^{\frac{2}{q}-2}$.
Indeed, we first observe that by taking $\lambda>0$ and replacing $(t,m)$ by $(\lambda\,t,\lambda\,m)$ we get
\[
F^*_q(-1,\xi)=\sup_{t>0,\,m> 0,\,\lambda>0} \left[-\lambda\,t+\lambda\,m\,|\xi|-\lambda^\frac{2}{q}\,m^2\,t^{\frac{2}{q}-2}\right].
\]
Now we observe that the derivative of the function
\[
h(\lambda)=-\lambda\,t+\lambda\,m\,|\xi|-\lambda^\frac{2}{q}\,m^2\,t^{\frac{2}{q}-2},
\]
is given by
\[
h'(\lambda)=-t+m\,|\xi|-\frac{2}{q}\,\lambda^{\frac{2}{q}-1}\,m^2\,t^{\frac{2}{q}-2}.
\]
We now distinguish two cases: if $m\,|\xi|-t\le 0$, the previous computation shows that $h$ is decreasing on $(0,+\infty)$ and thus
\[
h(\lambda)\le \lim_{\lambda\to0^+} h(\lambda)=0.
\]
On the other hand, if $m\,|\xi|-t >0$, then we get that $h$ has a unique maximum point at
\[
\lambda_0=\left(\frac{q}{2}\,\frac{m\,|\xi|-t}{m^2\,t^{\frac{2}{q}-2}}\right)^\frac{q}{2-q},
\]
thus 
\[
\begin{split}
h(\lambda)&\le\left(\frac{q}{2}\,\frac{m\,|\xi|-t}{m^2\,t^{\frac{2}{q}-2}}\right)^\frac{q}{2-q}\,\Big(-t+m\,|\xi|\Big)-\left(\frac{q}{2}\,\frac{m\,|\xi|-t}{m^2\,t^{\frac{2}{q}-2}}\right)^\frac{2}{2-q}\,m^2\,t^{\frac{2}{q}-2}\\
&=\left(\frac{m\,|\xi|-t}{m^q\,t^{1-q}}\right)^\frac{2}{2-q}\,\left(\frac{q}{2}\right)^\frac{q}{2-q}\,\frac{2-q}{2}.
\end{split}
\]
This discussion entails that
\[
F^*_q(-1,\xi)=\left(\frac{q}{2}\right)^\frac{q}{2-q}\,\frac{2-q}{2}\,\sup_{t> 0,\,m> 0} \left\{\left(\frac{m\,|\xi|-t}{m^q\,t^{1-q}}\right)^\frac{2}{2-q}\, :\, m\,|\xi|>t\right\}.
\]
We are left with computing such a supremum. We may notice that the objective function only depends on the ratio $t/m$, indeed we have
\[
\frac{m\,|\xi|-t}{m^q\,t^{1-q}}=\left(\frac{t}{m}\right)^q\,\left(\frac{m}{t}\,|\xi|-1\right).
\]
Thus, if we set $\tau=t/m$, we finally arrive at the problem
\[
F^*_q(-1,\xi)=\left(\frac{q}{2}\right)^\frac{q}{2-q}\,\frac{2-q}{2}\,\sup_{\tau> 0} \left\{\left(\tau^q\,\left(\frac{|\xi|}{\tau}-1\right)\right)^\frac{2}{2-q}\, :\, |\xi|>\tau\right\}.
\]
It is easily seen that the function
\[
f(\tau)=\tau^q\,\left(\frac{|\xi|}{\tau}-1\right),
\]
is maximal for 
\[
\tau_0=\frac{q-1}{q}\,|\xi|,
\]
which is an admissible value for $\tau$. Thus we obtain
\[
\sup_{\tau> 0} \left\{\tau^q\,\left(\frac{|\xi|}{\tau}-1\right)\, :\, |\xi|>\tau\right\}=\frac{1}{q}\,\left(\frac{q-1}{q}\right)^{q-1}\,|\xi|^q,
\]
which eventually leads to
\[
F^*_q(-1,\xi)=\left(\frac{q}{2}\right)^\frac{q}{2-q}\,\frac{2-q}{2}\,\left(\frac{1}{q}\,\left(\frac{q-1}{q}\right)^{q-1}\right)^\frac{2}{2-q}\,|\xi|^\frac{2\,q}{2-q}.
\]
Thanks to the positive homogeneity of $F_q^*$ already discussed, we get the desired conclusion.
\end{proof}
\begin{oss}[Relation between $F^*_q$ and $G_q$]
\label{oss:costanteriscala}
From the previous result, we get more generally that for every $C>0$, we have
\[
(C\,F_q)^*(s,\xi)=C\,F^*_q\left(\frac{s}{C},\frac{\xi}{C}\right)=C^{-\frac{q}{2-q}}\,F_q^*(s,\xi).
\]
This easily follows from the properties of the Legendre-Fenchel transform, together with the fact that $F_q^*$ is $2/(2-q)-$positively homogeneous. In particular, by taking $C=1/(2\,q)$, we have 
\[
\left(\frac{1}{2\,q}\,F_q\right)^*(s,\xi)=(2\,q)^\frac{q}{2-q}\,F_q^*(s,\xi).
\]
By recalling the definition \eqref{G} of $G_q$, we easily see that
\[
\Big(G_q(s,\xi)\Big)^\frac{2}{2-q}=\frac{1}{\alpha_q}\,F_q^*(s,\xi),
\]
and thus
\[
\left(\frac{1}{2\,q}\,F_q\right)^*(s,\xi)=(2\,q)^\frac{q}{2-q}\,F_q^*(s,\xi)=(2\,q)^\frac{q}{2-q}\,\alpha_q\, \Big(G_q(s,\xi)\Big)^\frac{2}{2-q}.
\]
Finally, by using that 
\[
\alpha_q=\frac{2-q}{2\,q}\,\left(\frac{q-1}{q}\right)^\frac{2\,(q-1)}{2-q}\,\left(\frac{1}{2}\right)^\frac{q}{2-q},
\]
we get the relation
\begin{equation}
\label{dualevero}
\left(\frac{1}{2\,q}\,F_q\right)^*(s,\xi)=\frac{2-q}{2}\,(q-1)^{(q-1)\,\frac{2}{2-q}}\, \Big(G_q(s,\xi)\Big)^\frac{2}{2-q}.
\end{equation}
We are going to use this identity in the proof of the main result.
\end{oss}

\section{A concave maximization problem}
\label{sec:4}

By combining Lemma \ref{lm:base} and the convexity of the function $F_q$ above, we can rewrite the variational problem which defines $\lambda_1(\Omega;q)$ as a concave optimization problem. This property is crucial for the proof of Theorem \ref{teo:sub}.
\begin{prop}
\label{prop:riformula}
Let $1<q<2$ and let $\Omega\subset\mathbb{R}^N$ be an open set with finite measure. We define the following subset of $W^{1,2}_0(\Omega)$
\[
X_q(\Omega)=\left\{\psi\in W^{1,2}_0(\Omega)\cap L^\infty(\Omega)\, :\, \int_\Omega F_q(\psi,\nabla \psi)\,dx<+\infty\right\}.
\]
Then $X_q(\Omega)$ is convex and we have
\begin{equation}
\label{uguali}
\begin{split}
\frac{2-q}{q}\,\left(\frac{1}{\lambda_1(\Omega;q)}\right)^{\frac{q}{2-q}}&=\max_{\varphi\in W^{1,2}_0(\Omega)}\left\{\frac{2}{q}\,\int_\Omega |\varphi|^q\,dx-\int_\Omega |\nabla \varphi|^2\,dx\right\}\\
&=\sup_{\psi\in X_q(\Omega)}\left\{\frac{2}{q}\,\int_\Omega \psi\,dx-\frac{1}{q^2}\,\int_\Omega F_q(\psi,\nabla \psi)\,dx\right\}.
\end{split}
\end{equation}
Finally, the last supremum is attained by a function $v\in X_q(\Omega)$ of the form 
\[
v=w^q,
\]
where $w$ is the same as in Lemma \ref{lm:base}.
\end{prop}
\begin{proof}
Convexity of $X_q(\Omega)$ immediately follows from the convexity of the function $F_q$. 
We now come to the proof of \eqref{uguali}.
The first identity is already contained in Lemma \ref{lm:base}. Let us take $w\in W^{1,2}_0(\Omega)\cap L^\infty(\Omega)$ to be the positive maximizer of 
\[
\max_{\varphi\in W^{1,2}_0(\Omega)}\left\{\frac{2}{q}\,\int_\Omega |\varphi|^q\,dx-\int_\Omega |\nabla \varphi|^2\,dx\right\}.
\]
We now set $v=w^q$ and observe that $v\in W^{1,2}_0(\Omega)\cap L^\infty(\Omega)$, since $v$ is the composition of a function in $W^{1,2}_0(\Omega)\cap L^\infty$ with a locally Lipschitz function, vanishing at the origin. From the chain rule in Sobolev spaces, we get
\[
\nabla v=q\,w^{q-1}\,\nabla w=q\,v^\frac{q-1}{q}\,\nabla w,
\]
where we also used the relation between $w$ and $v$, to replace $w^{q-1}$. Since $w>0$ in $\Omega$, we have the same property for $v$, as well. Thus we can infer
\[
\nabla w=\frac{1}{q}\,v^{\frac{1}{q}-1}\,\nabla v.
\]
By raising to the power $2$ and integrating, we get
\[
\int_\Omega |\nabla w|^2\,dx=\frac{1}{q^2}\,\int_\Omega |\nabla v|^2\,v^{\frac{2}{q}-2}\,dx=\frac{1}{q^2}\, \int_\Omega F_q(v,\nabla v)\,dx,
\]
which shows that $v\in X_q(\Omega)$. By recalling that $w$ is optimal, this also shows that 
\[
\max_{\varphi\in W^{1,2}_0(\Omega)}\left\{\frac{2}{q}\,\int_\Omega |\varphi|^q\,dx-\int_\Omega |\nabla \varphi|^2\,dx\right\}\le \sup_{\psi\in X_q(\Omega)}\left\{\frac{2}{q}\,\int_\Omega \psi\,dx-\frac{1}{q^2}\,\int_\Omega F_q(\psi,\nabla \psi)\,dx\right\}.
\]
On the other hand, let $\psi\in X_q(\Omega)$. Thanks to the form of the function $F_q$, this in particular implies that
\[
\psi(x)\ge 0,\qquad \mbox{ for a.\,e. }x\in\Omega. 
\]
For every $\varepsilon>0$, we introduce the $C^1$ function
\[
g_\varepsilon(\tau)=(\varepsilon^q+\tau)^\frac{1}{q}-\varepsilon,\qquad \mbox{ for every } \tau\ge 0.
\]
Then we set $\varphi_\varepsilon=g_\varepsilon\circ \psi$ and observe that $\varphi_\varepsilon\in W^{1,2}_0(\Omega)$, thanks to the fact that $g_\varepsilon$ is $C^1$ with bounded derivative and $g_\varepsilon(0)=0$. Again by the chain rule, we have
\[
\nabla \varphi_\varepsilon=g'_\varepsilon(\psi)\,\nabla \psi=\frac{1}{q}\,(\varepsilon^q+\psi)^{\frac{1}{q}-1}\,\nabla \psi.
\]
By integrating, we get
\[
\int_\Omega |\nabla \varphi_\varepsilon|^2\,dx=\frac{1}{q^2}\,\int_\Omega (\varepsilon^q+\psi)^{\frac{2}{q}-2}\,|\nabla \psi|^2\,dx\le \frac{1}{q^2}\,\int_\Omega F_q(\psi,\nabla \psi)\,dx. 
\]
In the last inequality, we used the well-known fact that $\nabla \psi$ vanishes almost everywhere on the zero set of the Sobolev function $\psi$ (see for example \cite[Theorem 6.19]{LL}).
This in turn implies that 
\[
\begin{split}
\max_{\varphi\in W^{1,2}_0(\Omega)}\left\{\frac{2}{q}\,\int_\Omega |\varphi|^q\,dx-\int_\Omega |\nabla \varphi|^2\,dx\right\}&\ge \frac{2}{q}\,\int_\Omega |\varphi_\varepsilon|^q\,dx-\int_\Omega |\nabla \varphi_\varepsilon|^2\,dx\\
&\ge \frac{2}{q}\,\int_\Omega g_\varepsilon(\psi)^q\,dx-\frac{1}{q^2}\,\int_\Omega F_q(\psi,\nabla \psi)\,dx.
\end{split}
\]
It is only left to pass to the limit as $\varepsilon$ goes to $0$ in the integral containing $g_\varepsilon(\psi)$. This can be done by a standard application of the Lebesgue Dominated Convergence Theorem. This finally leads to 
\[
\max_{\varphi\in W^{1,2}_0(\Omega)}\left\{\frac{2}{q}\,\int_\Omega |\varphi|^q\,dx-\int_\Omega |\nabla \varphi|^2\,dx\right\}\ge \frac{2}{q}\,\int_\Omega \psi\,dx-\frac{1}{q^2}\,\int_\Omega F_q(\psi,\nabla \psi)\,dx.
\]
By arbitrariness of $\psi\in X_q(\Omega)$, we eventually get the desired conclusion \eqref{uguali}.
\par
The above discussion also prove the last statement, about a maximizer of the problem settled over $X_q(\Omega)$.
\end{proof}
\begin{oss}
\label{oss:hidden}
The previous result is crucially based on the fact that the Dirichlet integral, apart from being convex in the usual sense, enjoys a suitable form of ``hidden'' convexity. In other words, for $\varphi$ positive we have that
\[
\varphi\mapsto \int_\Omega |\nabla \varphi|^2\,dx,
\] 
remains convex also with respect to the new variable $\psi=\varphi^q$, for the whole range $1\le q\le 2$. In the limit case $q=2$, this remarkable fact has been proved by Benguria, see \cite[Theorem 4.3]{Be1} and \cite[Lemma 4]{BBL}. For $1<q<2$ this property seems to have been first detected in \cite[Proposition 4]{Ka}, see also \cite[Proposition 1.1]{Na} and \cite[Example 5.2]{TTU}.
\end{oss}
Actually, we can restrict the maximization to smooth compactly supported functions, without affecting the value of the supremum. This is the content of the following result.
\begin{lm}
\label{lm:riformula}
With the notation of Proposition \ref{prop:riformula}, we have
\[
\sup_{\psi\in X_q(\Omega)}\left\{\frac{2}{q}\,\int_\Omega \psi\,dx-\frac{1}{q^2}\,\int_\Omega F_q(\psi,\nabla \psi)\,dx\right\}=\sup_{\psi\in X_q(\Omega)\cap C^1_0(\Omega)}\left\{\frac{2}{q}\,\int_\Omega \psi\,dx-\frac{1}{q^2}\,\int_\Omega F_q(\psi,\nabla \psi)\,dx\right\}.
\]
\end{lm}
\begin{proof}
We just need to prove that 
\[
\sup_{\psi\in X_q(\Omega)}\left\{\frac{2}{q}\,\int_\Omega \psi\,dx-\frac{1}{q^2}\,\int_\Omega F_q(\psi,\nabla \psi)\,dx\right\}\le \sup_{\psi\in X_q(\Omega)\cap C^1_0(\Omega)}\left\{\frac{2}{q}\,\int_\Omega \psi\,dx-\frac{1}{q^2}\,\int_\Omega F_q(\psi,\nabla \psi)\,dx\right\}.
\]
We take $\varphi\in C^\infty_0(\Omega)$ not identically zero and set $\psi=|\varphi|^q\in C^1_0(\Omega)$. As above, we have 
\[
|\nabla \psi|=q\,|\varphi|^{q-1}\,|\nabla \varphi|=q\,\psi^\frac{q-1}{q}\,|\nabla \varphi|,
\]
which holds everywhere on $\Omega$. This in particular implies that $\nabla \psi$ vanishes on every point where $\psi$ vanishes. Thus we have 
\[
F_q(\psi,\nabla \psi)=\left\{\begin{array}{rl}
|\nabla\psi|^2\,\psi^{\frac{2}{q}-2},& \mbox{ if } \psi\not=0,\\
0,& \mbox{ if }\psi=0.
\end{array}
\right.
\]
By integrating and recalling the relation above between $\nabla\psi$, $\psi$ and $\nabla\varphi$, we then obtain
\[
\int_\Omega F_q(\psi,\nabla\psi)\,dx\le q^2\,\int_\Omega |\nabla \varphi|^2\,dx.
\]
This in turn implies
\[
\frac{2}{q}\,\int_\Omega \psi\,dx-\frac{1}{q^2}\,\int_\Omega F_q(\psi,\nabla\psi)\,dx\ge \frac{2}{q}\,\int_\Omega |\varphi|^q\,dx-\int_\Omega |\nabla \varphi|^2\,dx.
\]
By arbitrariness of $\varphi\in C^\infty_0(\Omega)$, we get
\begin{equation}
\label{prot}
\begin{split}
\sup_{\psi\in X_q(\Omega)\cap C^1_0(\Omega)}&\left\{\frac{2}{q}\,\int_\Omega \psi\,dx-\frac{1}{q^2}\,\int_\Omega F_q(\psi,\nabla \psi)\,dx\right\}\\
&\ge \sup_{\varphi\in C^\infty_0(\Omega)}\left\{\frac{2}{q}\,\int_\Omega |\varphi|^q\,dx-\int_\Omega |\nabla \varphi|^2\,dx\right\}.
\end{split}
\end{equation}
On the other hand, by density of $C^\infty_0(\Omega)$ in $W^{1,2}_0(\Omega)$, it is easily seen that 
\[
\sup_{\varphi\in C^\infty_0(\Omega)}\left\{\frac{2}{q}\,\int_\Omega |\varphi|^q\,dx-\int_\Omega |\nabla \varphi|^2\,dx\right\}=\max_{\varphi\in W^{1,2}_0(\Omega)}\left\{\frac{2}{q}\,\int_\Omega |\varphi|^q\,dx-\int_\Omega |\nabla \varphi|^2\,dx\right\}.
\]
The desired conclusion now follows by combining Proposition \ref{prop:riformula} and \eqref{prot}.
\end{proof}

\section{Proof of the main results}
\label{sec:5}

\subsection{Proof of Theorem \ref{teo:sub}}

We recall the definition
\[
\mathcal{A}(\Omega)=\Big\{(f,\phi)\in L^1_{\rm loc}(\Omega)\times L^2_{\rm loc}(\Omega;\mathbb{R}^N)\, :\, -\mathrm{div\,} \phi+f\ge 1 \mbox{ in }\Omega\Big\},
\]
where the condition $-\mathrm{div\,} \phi+f\ge 1$ has to be intended in distributional sense, i.\,e.
\[
\int_\Omega \Big[\langle \nabla \psi,\phi\rangle+f\,\psi\Big]\,dx\ge \int_\Omega \psi\,dx,\qquad \mbox{ for every } \psi\in C^1_0(\Omega) \mbox{ such that } \psi\ge 0.
\]
In particular, for every $(f,\phi)\in\mathcal{A}(\Omega)$ and every $\psi\in X_q(\Omega)\cap C^1_0(\Omega)$, we can write
\[
\begin{split}
\frac{2}{q}\,\int_\Omega \psi\,dx-\frac{1}{q^2}\,\int_\Omega F_q(\psi,\nabla \psi)\,dx&\le \frac{2}{q}\,\int_\Omega \Big[f\,\psi+\langle \nabla \psi,\phi\rangle\Big]\,dx-\frac{1}{q^2}\,\int_\Omega F_q(\psi,\nabla \psi)\,dx\\
&=\frac{2}{q}\,\int_\Omega \left[f\,\psi+\langle \nabla \psi, \phi\rangle-\frac{1}{2\,q}\,F_q(\psi,\nabla \psi)\right]\,dx.
\end{split}
\]
By Lemma \ref{lm:fenchel} and equation \eqref{dualevero} from Remark \ref{oss:costanteriscala}, the following inequality holds almost everywhere
\[
f\,\psi+\langle \nabla \psi,\nabla \phi\rangle-\frac{1}{2\,q}\,F_q(\psi,\nabla \psi)\le \frac{2-q}{2}\,(q-1)^{(q-1)\,\frac{2}{2-q}}\,\Big(G_q(f,\phi)\Big)^\frac{2}{2-q}.
\]
This simply follows from the definition of Legendre-Fenchel transform.
By integrating this inequality and taking the supremum over $\psi$, we obtain
\[
\begin{split}
\sup_{\psi\in X_q(\Omega)\cap C^1_0(\Omega)}&\left\{\frac{2}{q}\,\int_\Omega \psi\,dx-\frac{1}{q^2}\,\int_\Omega F_q(\psi,\nabla \psi)\,dx\right\}\le \frac{2-q}{q}\,(q-1)^{(q-1)\,\frac{2}{2-q}}\,\int_\Omega \Big(G_q(f,\phi)\Big)^\frac{2}{2-q}\,dx.
\end{split}
\]
By combining Proposition \ref{prop:riformula} and Lemma \ref{lm:riformula} and taking the infimum over admissible pairs $(f,\phi)$, we get
\[
\frac{1}{\lambda_1(\Omega;q)}\le (q-1)^{(q-1)\frac{2}{q}}\,\inf_{(f,\phi)\in \mathcal{A}(\Omega)}\left(\int_\Omega \Big(G_q(f,\phi)\Big)^\frac{2}{2-q}\,dx\right)^\frac{2-q}{q}.
\]
In order to prove the reverse inequality and identify a minimizing pair, we take 
\[
\phi_0=\frac{\nabla w}{w^{q-1}}\qquad \mbox{ and }\qquad f_0=-(q-1)\,\frac{|\nabla w|^2}{w^{q}},
\]
where $w$ is the function of Lemma \ref{lm:base}. In light of the properties of $w$, both $\phi_0$ and $f_0$ have the required integrability properties.
Moreover, it is not difficult to see that these are admissible, since it holds
\[
-\mathrm{div\,}\phi_0+f_0=1,\qquad \mbox{ in }\Omega,
\]
in distributional sense.
It is sufficient to use the equation solved by $w$.
By recalling the definition \eqref{G} of $G_q$, when $\nabla w\not =0$ we have 
\[
G_q(f_0,\phi_0)=|\phi_0|^q\,|f_0|^{1-q}=(q-1)^{1-q}\,|\nabla w|^{2-q}.
\]
On the other hand, since both $f_0$ and $\phi_0$ vanish when $\nabla w =0$, in this case we  have that $G_q(f_0,\phi_0)$ would vanish, as well. In conclusion, we get
\[
\int_\Omega \Big(G_q(f_0,\phi_0)\Big)^\frac{2}{2-q}\,dx=(q-1)^\frac{2\,(1-q)}{2-q}\,\int_\Omega |\nabla w|^2\,dx.
\]
Thus we obtain 
\[
\begin{split}
(q-1)^{(q-1)\frac{2}{q}}\,\inf_{(f,\phi)\in \mathcal{A}(\Omega)}\left(\int_\Omega \Big(G_q(f_0,\phi_0)\Big)^\frac{2}{2-q}\,dx\right)^\frac{2-q}{q}\le \left(\int_\Omega |\nabla w|^2\,dx\right)^\frac{2-q}{q}.
\end{split}
\]
We can now use that by Lemma \ref{lm:base}
\[
\begin{split}
\int_\Omega |\nabla w|^2\,dx&=\frac{q}{2-q}\,\left(\frac{2}{q}\,\int_\Omega |w|^q\,dx-\int_\Omega|\nabla w|^2\,dx\right)\\
&=\frac{q}{2-q}\,\max_{\varphi\in W^{1,2}_0(\Omega)}\left\{\frac{2}{q}\,\int_\Omega |\varphi|^q\,dx-\int_\Omega |\nabla \varphi|^2\,dx\right\}=\left(\frac{1}{\lambda_1(\Omega;q)}\right)^\frac{q}{2-q}.
\end{split}
\]
The first identity above follows by testing the Euler-Lagrange equation for $w$, i.e.
\[
\int_\Omega \langle \nabla w,\nabla\varphi\rangle\,dx=\int_\Omega w^{q-1}\,\varphi\,dx,\qquad \mbox{ for every } \varphi\in W^{1,2}_0(\Omega),
\]
with $\varphi=w$ itself.
This finally proves that the reverse inequality holds
\[
(q-1)^{(q-1)\frac{2}{q}}\,\inf_{(f,\phi)\in \mathcal{A}(\Omega)}\left(\int_\Omega \Big(G_q(f_0,\phi_0)\Big)^\frac{2}{2-q}\,dx\right)^\frac{2-q}{q}\le \frac{1}{\lambda_1(\Omega;q)},
\]
as well. The second part of the proof also proves that $(f_0,\phi_0)$ is an optimal pair. 
\subsection{Proof of Theorem \ref{teo:homo}}
In order to prove the inequality
\begin{equation}
\label{l1}
\frac{1}{\lambda_1(\Omega)}\le \inf_{(f,\phi)\in\mathcal{A}(\Omega)} \Big\|G_2(f,\phi)\Big\|_{L^\infty(\Omega)},
\end{equation}
we will go through a limiting argument, for simplicity. Let $(f,\phi)\in \mathcal{A}(\Omega)$ be an admissible pair, we can suppose that 
\[
M:=\Big\|G_2(f,\phi)\Big\|_{L^\infty(\Omega)}<+\infty.
\]
Thanks to the definition of $G_2$, this implies in particular that the pair $(f,\phi)$ has the following properties:
\[
|f(x)|=0\quad \mbox{ implies that }\quad |\phi(x)|=0,\qquad \mbox{ for a.\,e. }x\in\Omega,
\]
and
\[
|\phi(x)|^2\le M\,|f(x)|,\qquad \mbox{ for a.\,e. }x\in \Omega.
\]
We thus obtain for every $1<q<2$
\[
\frac{|\phi|^q}{|f|^{q-1}}\le M^\frac{q}{2}\,|f|^\frac{2-q}{2}\in L^\frac{2}{2-q}_{\rm loc}(\Omega),
\]
and thus 
\[
\Big(G_q(f,\phi)\Big)^\frac{2}{2-q}\le M\,|f|\in L^1_{\rm loc}(\Omega).
\]
This estimate guarantees that for every open set $\Omega'$ compactly contained in $\Omega$, we have
\begin{equation}
\label{domenica}
\left(\int_{\Omega'} \Big(G_q(f,\phi)\Big)^\frac{2}{2-q}\,dx\right)^\frac{2-q}{q}\le M\,\left(\int_{\Omega'} |f|\,dx\right)^\frac{2-q}{q}= \Big\|G_2(f,\phi)\Big\|_{L^\infty(\Omega)}\, \left(\int_{\Omega'} |f|\,dx\right)^\frac{2-q}{q}.
\end{equation}
By Theorem \ref{teo:sub} applied to $\Omega'$, we have for every $1<q<2$
\[
\frac{1}{\lambda_1(\Omega';q)}\le (q-1)^{(q-1)\frac{2}{q}}\, \Big\|G_q(f,\phi)\Big\|_{L^\frac{2}{2-q}(\Omega')}^\frac{2}{q}.
\]
By using \eqref{domenica} on the right-hand side and taking the limit as $q$ goes to $2$, we get
\begin{equation}
\label{dentro}
\frac{1}{\lambda_1(\Omega')}\le  \Big\|G_2(f,\phi)\Big\|_{L^\infty(\Omega)},
\end{equation}
by virtue of Lemma \ref{lm:limite}.
We can now take an increasing sequence of open sets $\{\Omega_n\}_{n\in\mathbb{N}}$ compactly contained in $\Omega$ and invading it, i.e. such that 
\[
\Omega=\bigcup_{n\in\mathbb{N}} \Omega_n.
\] 
By using that\footnote{This simply follows from the properties of the sequence $\{\Omega_n\}_{n\in\mathbb{N}}$ and the fact that
\[
\lambda_1(\Omega)=\min_{\varphi\in W^{1,2}_0(\Omega)\setminus\{0\}} \frac{\displaystyle \int_\Omega |\nabla \varphi|^2\,dx}{\displaystyle\int_\Omega |\varphi|^2\,dx}=\inf_{\varphi\in C^\infty_0(\Omega)\setminus\{0\}} \frac{\displaystyle \int_\Omega |\nabla \varphi|^2\,dx}{\displaystyle\int_\Omega |\varphi|^2\,dx}.
\]} 
\[
\lim_{n\to\infty}\lambda_1(\Omega_n)=\lambda_1(\Omega),
\]
and applying \eqref{dentro} to each $\Omega_n$, we get
\[
\frac{1}{\lambda_1(\Omega)}\le  \Big\|G_2(f,\phi)\Big\|_{L^\infty(\Omega)}.
\]
By arbitrariness of $(f,\phi)\in\mathcal{A}(\Omega)$, we get \eqref{l1} as desired.
\par
In order to prove the reverse inequality, we take
\[
\phi_0=\frac{1}{\lambda_1(\Omega)}\,\frac{\nabla U}{U}\qquad \mbox{ and }\qquad f_0=-\frac{1}{\lambda_1(\Omega)}\,\frac{|\nabla U|^2}{U^2},
\]
where $U$ is any positive first eigenfunction of $\Omega$.
It is easily seen that this pair is admissible for the variational problem 
\[
\inf_{(f,\phi)\in \mathcal{A}(\Omega)} \Big\|G_2(f,\phi)\Big\|_{L^\infty(\Omega)}.
\]
Moreover, we have 
\[
G_2(f_0,\phi_0)\le \frac{1}{\lambda_1(\Omega)},\qquad \mbox{ a.\,e. in }\Omega,
\]
where we also used that $f_0$ vanishes if and only if $\phi_0$ vanishes and in this case $G_2(f_0,\phi_0)=0$. This gives
\[
\begin{split}
\inf_{(f,\phi)\in \mathcal{A}(\Omega)} \Big\|G_2(f,\phi)\Big\|_{L^\infty(\Omega)}\le \Big\|G_2(f_0,\phi_0)\Big\|_{L^\infty(\Omega)}=\frac{1}{\lambda_1(\Omega)},
\end{split}
\]
thus concluding the proof.

\section{Applications to geometric estimates}
\label{sec:6}

In this section, we briefly sketch some geometric estimates for the generalized principal frequencies, that can be inferred from our main result.

\subsection{Diaz-Weinstein--type estimates}
We start by recalling the {\it Diaz-Weinstein inequality} for the torsional rigidity. This is given by the following estimate\footnote{In \cite{DW} the case $N=2$ is considered and a slightly different proof is given. The definition of torsional rigidity in \cite{DW} coincides with ours for simply connected sets, up to a multiplicative factor $4$.}
\begin{equation}
\label{DW}
T(\Omega)\le \frac{1}{N^2}\,\mathcal{I}_2(\Omega),\qquad \mbox{ where }\quad \mathcal{I}_2(\Omega)=\min\limits_{x_0\in\mathbb{R}^N}\int_\Omega |x-x_0|^2\,dx,
\end{equation}
see \cite[formula (11)]{DW}, which is valid for open sets $\Omega\subset\mathbb{R}^N$ such that
\[
\int_\Omega |x|^2\,dx<+\infty.
\]
The quantity $\mathcal{I}_2(\Omega)$ is sometimes called {\it polar moment of inertia} of $\Omega$. It is easily seen that the minimum in its definition is uniquely attained at the {\it centroid} of $\Omega$, i.e. at the point
\[
x_\Omega=\fint_\Omega x\,dx.
\]
The Diaz-Weinstein inequality can be proved by appealing to the dual formulation for the torsional rigidity. Indeed, for every $x_0\in\Omega$, it is sufficient to use the admissible vector field 
\[
\phi_0=\frac{x_0-x}{N},
\]
in the dual problem \eqref{dualT}. This automatically gives
\[
T(\Omega)\le \frac{1}{N^2}\,\int_\Omega |x-x_0|^2\,dx,
\]
and thus \eqref{DW} follows by arbitrariness of $x_0\in\mathbb{R}^N$.
We observe that this estimate is sharp, as equality is attained for a ball. Indeed, recall that the unique $W^{1,2}_0(\Omega)$ solution of $-\Delta u=1$ in a ball of radius $R$ and center $x_0$ is given by
\[
w(x)=\frac{R^2-|x-x_0|^2}{2\,N},\qquad \mbox{ for every $x\in\mathbb{R}^N$ such that }|x-x_0|<R.
\]
Thus, by observing that $\phi_0=\nabla w$ and recalling the discussion in Subsection \ref{sec:torsion}, we get the claimed optimality.
\vskip.2cm\noindent
We now show how the previous argument can be extended to the case $1<q<2$.
We fix again a point $x_0\in\mathbb{R}^N$ and take a constant $\alpha>1/N$, then we choose the pair
\[
\phi_0(x)=\alpha\,(x_0-x)\qquad \mbox{ and }\qquad f_0(x)=1-\alpha\,N.
\]
Observe that this solves 
\[
-\mathrm{div\,}\phi_0+f_0=1,\qquad \mbox{ in } \mathbb{R}^N.
\]
Thus the pair $(f_0,\phi_0)$ is admissible for the dual problem \eqref{formulab1}, for every $1<q<2$.
By Theorem \ref{teo:sub} we immediately get
\[
\begin{split}
\frac{1}{\lambda_1(\Omega;q)}&\le (q-1)^{(q-1)\frac{2}{q}}\, \Big\|G_q(f_0,\phi_0)\Big\|_{L^\frac{2}{2-q}(\Omega)}^\frac{2}{q}\\
&=(q-1)^{(q-1)\frac{2}{q}}\,\frac{\alpha^2}{(\alpha\,N-1)^{(q-1)\frac{2}{q}}}\,\left(\int_\Omega |x-x_0|^\frac{2\,q}{2-q}\,dx\right)^\frac{2-q}{q}.
\end{split}
\]
We now observe that the quantity
\[
\frac{\alpha^2}{(\alpha\,N-1)^{(q-1)\frac{2}{q}}},
\]
is minimal for $\alpha=q/N$. By making such a choice for $\alpha$ and using the
arbitrariness of $x_0$, we then get the following
\begin{coro}[Diaz-Weinstein--type estimate]
Let $1<q<2$ and let $\Omega\subset\mathbb{R}^N$ be an open set such that
\[
\int_\Omega |x|^\frac{2\,q}{2-q}\,dx<+\infty.
\]
Then we have
\begin{equation}
\label{moment}
\lambda_1(\Omega;q)\ge \left(\frac{q}{N}\right)^2\,\Big(\mathcal{I}_\frac{2\,q}{2-q}(\Omega)\Big)^{-\frac{2-q}{q}},\qquad \mbox{ where }\quad \mathcal{I}_\frac{2\,q}{2-q}(\Omega)=\min\limits_{x_0\in\mathbb{R}^N}\int_\Omega |x-x_0|^\frac{2\,q}{2-q}\,dx.
\end{equation}
\end{coro}
\begin{oss}
We notice that for $1<q<2$ the estimate \eqref{moment} does not appear to be sharp. In order to get the sharp constant, it seems unavoidable the use of more sophisticated arguments, based on {\it radially symmetric decreasing rearrangements}.
These permit to show that both quantities
\[
\lambda_1(\Omega;q)\qquad \mbox{ and }\qquad \mathcal{I}_\frac{2\,q}{2-q}(\Omega),
\]
are minimal for a ball, among sets with given measure. By combining these two facts, then we get that the sharp constant in \eqref{moment} is given by
\[
\lambda_1(B_1;q)\, \left(\mathcal{I}_\frac{2\,q}{2-q}(B_1)\right)^{\frac{2-q}{q}},
\]
where $B_1=\{x\in\mathbb{R}^N\, :\, |x|<1\}$.
\par
Nevertheless, we believe that the duality-based proof exposed above is interesting anyway: this gives a cheap way to get a scale invariant geometric estimate with a simple explicit constant, by means of an elementary argument.
\end{oss}


\subsection{Cheeger--type estimates}
We recall that the {\it Cheeger constant} for an open set $\Omega\subset\mathbb{R}^N$ is given by
\[
h_1(\Omega)=\inf\left\{\frac{P(E)}{|E|}\, :\, E\subset \Omega \mbox{ bounded} \mbox{ with } |E|>0\right\}.
\]
Here $P(E)$ is the distributional perimeter of a set $E$. The Cheeger constant has the following dual characterization
\begin{equation}
\label{cheegerdual}
\frac{1}{h_1(\Omega)}=\min_{\phi\in L^\infty(\Omega;\mathbb{R}^N)}\Big\{\|\phi\|_{L^\infty(\Omega)}\, :\, -\mathrm{div\,}\phi=1 \mbox{ in }\Omega\Big\}.
\end{equation}
This characterization seems to have first appeared in \cite[Section 4]{St}. The fact that the minimum in \eqref{cheegerdual} is attained easily follows from the Direct Method in the Calculus of Variations.
\par
We take an optimal vector field $\phi_\Omega$ in \eqref{cheegerdual} and then for every $\varepsilon>0$ we make the choice
\[
\phi_0=(1+\varepsilon)\,\phi_\Omega\qquad \mbox{ and }\qquad f_0=-\varepsilon.
\]
By observing that $(f_0,\phi_0)\in\mathcal{A}(\Omega)$, from Theorem \ref{teo:sub} we get
\[
\begin{split}
\frac{1}{\lambda_1(\Omega;q)}&\le (q-1)^{(q-1)\frac{2}{q}}\,\Big\|G_q(f_0,\phi_0)\Big\|_{L^\frac{2}{2-q}(\Omega)}^\frac{2}{q}\\
&=(q-1)^{(q-1)\frac{2}{q}}\,\frac{(1+\varepsilon)^2}{\varepsilon^{(q-1)\,\frac{2}{q}}}\,\left(\int_\Omega |\phi_\Omega|^\frac{2\,q}{2-q}\,dx\right)^\frac{2-q}{q}\\
&\le (q-1)^{(q-1)\frac{2}{q}}\,\frac{(1+\varepsilon)^2}{\varepsilon^{(q-1)\,\frac{2}{q}}}\,\|\phi_\Omega\|_{L^\infty(\Omega)}^2\,|\Omega|^\frac{2-q}{q}.
\end{split}
\]
We now notice that the quantity
\[
\frac{(1+\varepsilon)^2}{\varepsilon^{(q-1)\,\frac{2}{q}}},
\]
is minimal with the choice $\varepsilon=(q-1)$. Then such a choice leads to the estimate
\[
\frac{1}{\lambda_1(\Omega;q)}\le q^2\,\frac{1}{h_1(\Omega)^2}\,|\Omega|^\frac{2-q}{q}.
\]
Thus, we proved the following
\begin{coro}
Let $1<q<2$ and let $\Omega\subset\mathbb{R}^N$ be an open set, with finite measure.
Then we have the Cheeger-type inequality
\begin{equation}
\label{cheeger}
\left(\frac{h_1(\Omega)}{q}\right)^2\le |\Omega|^\frac{2-q}{q}\,\lambda_1(\Omega;q).
\end{equation}
\end{coro}
\begin{oss}
By taking the limits as $q$ goes to $1$ and as $q$ goes to $2$ in \eqref{cheeger}, we recover 
\[
h_1(\Omega)^2\le \frac{|\Omega|}{T(\Omega)} \qquad \mbox{ and }\qquad \left(\frac{h_1(\Omega)}{2}\right)^2\le\lambda_1(\Omega),
\]
respectively. The first estimate has been proved in \cite[Theorem 2]{BE}, while the second one is the classical {\it Cheeger inequality} for the Laplacian, see \cite{cheeger}. Both inequalities are sharp in the following sense: by taking the $N-$dimensional unit ball $B_1(0)$, one may prove that 
\[
\lim_{N\to\infty} \frac{|B_1(0)|}{T(B_1(0))}\,\frac{1}{h_1(B_1(0))^2}=1\qquad \mbox{ and }\qquad \lim_{N\to\infty} \frac{\lambda_1(B_1(0))}{h_1(B_1(0))^2}=\frac{1}{4}.
\]
The first fact can be easily seen, by recalling that
\[
T(B_1(0))=\frac{|B_1(0)|}{N\,(N+2)}\qquad \mbox{ and }\qquad h_1(B_1(0))=N.
\]
The second fact has been recently observed in \cite[Theorem 1.3]{Ft} and is based on asymptotics for zeros of Bessel functions. 
\par
This somehow suggests that the general estimate \eqref{cheeger} should be sharp, as well, by using a similar argument. However, the task of computing the exact asymptotics for $\lambda_1(B_1(0);q)$, as the dimension $N$ goes to $\infty$, does not seem easy.
\par
We refer to \cite{Le} for some recent progress on Cheeger--type inequalities.
\end{oss}

\subsection{Hersch-Makai--type estimates} We now take $\Omega\subset\mathbb{R}^N$ to be an open bounded convex set. We will employ in a dual way a trick by Kajikiya (see \cite{Kaj, Kaj2} and also \cite{BM}), in conjunction with our duality result. This will give us a sharp lower bound on $\lambda_1(\Omega;q)$ in terms on the inradius and the perimeter of the set.
\vskip.2cm\noindent
We indicate by $d_\Omega$ the distance function from the boundary $\partial\Omega$, while $R_\Omega$ will be the {\it inradius} of $\Omega$. We recall that this coincides with the supremum of the distance function.
\par
We take $g$ to be the unique positive solution of
\[
\max_{\varphi\in W^{1,2}_0((-1,1))}\left\{\frac{2}{q}\,\int_{-1}^1 |\varphi|^q\,dt-\int_{-1}^1 |\varphi'|^2\,dt\right\}.
\]
This satisfies the equation
\[
-g''=g^{q-1},\quad \mbox{ in } (-1,1),\qquad \mbox{ with } g(-1)=g(1)=0.
\]
This is a concave even function, which is increasing on $(-1,0)$ and decreasing on $(0,1)$.
We then ``transplant'' this function to $\Omega$, by setting
\[
u(x)=R_\Omega^\frac{2}{2-q}\,g\left(\frac{d_\Omega(x)}{R_\Omega}-1\right),\qquad \mbox{ for } x\in\Omega.
\]
By using the equation solved by $g$, the fact that $|\nabla d_\Omega|=1$ almost everywhere and the weak superharmonicity\footnote{This follows from the convexity of $\Omega$.} of $d_\Omega$, we get that
\[
-\Delta u\ge u^{q-1},\qquad \mbox{ in }\Omega,
\]
in weak sense. This entails that the pair
\[
\phi_0=\frac{\nabla u}{u^{q-1}}\qquad \mbox{ and }\qquad f_0=-(q-1)\,\frac{|\nabla u|^2}{u^q},
\]
is admissible for the dual problem \eqref{formulab1}. By Theorem \ref{teo:sub}, we then get
\[
\frac{1}{\lambda_1(\Omega;q)}\le (q-1)^{(q-1)\frac{2}{q}}\,\Big\|G_q(f_0,\phi_0)\Big\|_{L^\frac{2}{2-q}(\Omega)}^\frac{2}{q}=\left( \int_\Omega |\nabla u|^2\,dx\right)^\frac{2-q}{q}.
\]
By using the explicit form of $u$ and again the fact that $|\nabla d_\Omega|=1$ almost everywhere in $\Omega$, the previous estimate can be rewritten as
\begin{equation}
\label{ziooo}
\frac{1}{\lambda_1(\Omega;q)}\le R^2_\Omega\,\left(\int_\Omega \left|g'\left(\frac{d_\Omega}{R_\Omega}-1\right)\right|^2\,dx\right)^\frac{2-q}{q}.
\end{equation}
We observe that this is already a geometric estimate in nuce, since the right-hand side only depends on elementary geometric quantities of $\Omega$ (i.e. the distance function and the inradius) and on the universal one-dimensional function $g$. Moreover, it is not difficult to see that \eqref{ziooo} is sharp (see Remark \ref{oss:HM} below).
\par
Let us try to derive from \eqref{ziooo} a more explicit estimate.
At this aim, we can use the Coarea Formula with respect to the distance function, so to get
\[
\int_\Omega \left|g'\left(\frac{d_\Omega}{R_\Omega}-1\right)\right|^2\,dx=\int_0^{R_\Omega} \left|g'\left(\frac{t}{R_\Omega}-1\right)\right|^2\,P(\Omega_t)\,dt,
\]
where $\Omega_t=\{x\in\Omega\, :\,d_\Omega(x)>t\}$. We now recall that $t\mapsto P(\Omega_t)$ is monotone decreasing, in a convex set (see \cite[Lemma 2.2.2]{BB}). Thus we automatically get
\[
\int_\Omega \left|g'\left(\frac{d_\Omega}{R_\Omega}-1\right)\right|^2\,dx\le P(\Omega)\,\int_0^{R_\Omega} \left|g'\left(\frac{t}{R_\Omega}-1\right)\right|^2\,dt.
\]
A simple change of variable then leads to
\[
\int_\Omega \left|g'\left(\frac{d_\Omega}{R_\Omega}-1\right)\right|^2\,dx\le P(\Omega)\,R_\Omega\,\int_{-1}^{0} |g'(\tau)|^2\,d\tau.
\]
We can insert this estimate in \eqref{ziooo} to obtain
\[
\frac{1}{\lambda_1(\Omega;q)}\le R_\Omega^\frac{q+2}{q}\,P(\Omega)^\frac{2-q}{q}\,\left(\int_{-1}^0 |g'(\tau)|^2\,d\tau\right)^\frac{2-q}{q}.
\]
By using that $g$ is even and the identity
\[
\int_{-1}^1 |g'(\tau)|^2\,d\tau=\int_{-1}^1 |g(\tau)|^q\,d\tau,
\]
we have 
\[
\begin{split}
\int_{-1}^0 |g'|^2\,d\tau=\frac{1}{2}\,\int_{-1}^1 |g'|^2\,d\tau&=\frac{1}{2}\,\frac{q}{2-q}\,\left[\frac{2}{q}\,\int_{-1}^1 |g|^q\,d\tau-\int_{-1}^1 |g'|^2\,d\tau\right]\\
&=\frac{1}{2}\,\left(\frac{1}{\lambda_1((-1,1);q)}\right)^{\frac{q}{2-q}}
\end{split}
\]
In the last equality we used the optimality of $g$ and Proposition \ref{lm:base}, for the one-dimensional set $\Omega=(-1,1)$.
This gives
\[
\frac{1}{\lambda_1(\Omega;q)}\le R_\Omega^\frac{q+2}{q}\,P(\Omega)^\frac{2-q}{q}\,\left(\frac{1}{2}\right)^\frac{2-q}{q}\,\frac{1}{\lambda_1((-1,1);q)}.
\]
If we now recall the definition 
\[
\pi_{2,q}=\inf_{\varphi\in W^{1,2}_0((0,1))\setminus\{0\}}\frac{\|\varphi'\|_{L^2((0,1))}}{\|\varphi\|_{L^q((0,1))}},
\]
and use the scaling properties of Sobolev-Poincar\'e constants, we get
\[
\lambda_1((-1,1);q)=(\pi_{2,q})^2\,2^{-\frac{2+q}{q}}.
\]
Thus, we finally obtain the following
\begin{coro}
Let $1<q<2$ and let $\Omega\subset\mathbb{R}^N$ be an open bounded convex set. Then we have
\begin{equation}
\label{HM}
\lambda_1(\Omega;q)\ge \left(\frac{\pi_{2,q}}{2}\right)^2\,\frac{P(\Omega)^\frac{q-2}{q}}{R_\Omega^\frac{q+2}{q}}.
\end{equation}
\end{coro}
\begin{oss}
\label{oss:HM}
When compared with the {\it Hersch-Makai inequality} 
\begin{equation}
\label{HMtrue}
\lambda_1(\Omega;q)\ge \left(\frac{\pi_{2,q}}{2}\right)^2\, \frac{|\Omega|^\frac{q-2}{q}}{R_\Omega^2},
\end{equation}
already recalled in the Introduction, we see that the estimate \eqref{HM} is slightly weaker. Indeed, the former implies the latter, by recalling that for a open bounded convex set we have
\[
|\Omega|\le R_\Omega\,P(\Omega).
\]
Nevertheless, inequality \eqref{HM} {\it is still sharp}: it is sufficient to take
the ``slab--type'' sequence 
\[
\Omega_L=\left(-\frac{L}{2},\frac{L}{2}\right)^{N-1}\times(0,1),
\] 
with $L$ diverging to $+\infty$. For this family of sets we have (see \cite[Lemma A.2]{BM})
\[
\lambda_1(\Omega_L;q)\sim \frac{(\pi_{2,q})^2}{L^{(N-1)\,\frac{2-q}{q}}}\quad \mbox{ and }\quad P(\Omega_L)\sim 2\,L^{N-1},\qquad \mbox{ as }L\to +\infty,
\]
and $R_{\Omega_L}=1/2$, for $L>1$.
\par
We also observe that this slight discrepancy between \eqref{HM} and \eqref{HMtrue} is lost in the limit as $q$ converges to $2$: in both cases the estimates boil down to 
\[
\lambda_1(\Omega)\ge \left(\frac{\pi}{2}\right)^2\, \frac{1}{R_\Omega^2},
\]
which is the original Hersch sharp inequality from \cite[Th\'eor\`eme 8.1]{He2}. We also refer to \cite[Theorem 5.5]{BuGuMa}, \cite[Theorem 5.1]{DDG} and \cite[Theorem 2.1]{Kaj} for other proofs and extensions of this result.
\end{oss}

\end{document}